\documentclass[a4paper,reqno]{amsart}
\usepackage{amsmath,amssymb,amsthm,graphicx,mathrsfs,url}
\usepackage[usenames,dvipsnames]{color}
\definecolor{darkred}{rgb}{0.4,0.1,0.1}
\definecolor{darkblue}{rgb}{0.1,0.1,0.4}

\usepackage{tikz}
\usetikzlibrary{datavisualization}
\usetikzlibrary{datavisualization.formats.functions}
\usetikzlibrary{hobby,backgrounds,patterns}

\usepackage[normalem]{ulem}

\numberwithin{equation}{section}
\theoremstyle{plain}

\newtheorem{theorem}{Theorem}[section]

\newtheorem{proposition}[theorem]{Proposition}
\newtheorem{corollary}[theorem]{Corollary}

\theoremstyle{remark}
\newtheorem{remark}[theorem]{Remark}

\theoremstyle{definition}

\newcommand\cH{\mathcal H}

\DeclareMathOperator{\diver}{div}
\DeclareMathOperator{\curl}{curl}

\definecolor{darkgreen}{rgb}{0.1,0.45,0.1}
\definecolor{darkblue}{rgb}{0.1,0.1,0.4}
\definecolor{darkgrey}{rgb}{0.5,0.5,0.5}

\definecolor{darkred}{rgb}{0.6,0.0,0.0}

\DeclareMathOperator\ran{ran}

\newcommand\void[1]{}

\renewcommand{\phi}{\varphi}

\def\sb{\mathfrak b}

   \def\cH{{\mathcal H}}

\def\R{\mathbb{R}}

\def\N{\mathbb{N}}

\newcommand{\dom}{\mathrm{dom}\,}

\allowdisplaybreaks

\newcounter{counter_a}

\title[Curl curl versus Dirichlet Laplacian eigenvalues]{Curl curl versus Dirichlet Laplacian eigenvalues}

\author[J.~Rohleder]{Jonathan Rohleder}
\address{Matematiska institutionen \\ Stockholms universitet \\
106 91 Stockholm \\
Sweden}
\email{jonathan.rohleder@math.su.se}



\begin{document}

\begin{abstract}
We provide an upper estimate for the eigenvalues of the curl curl operator on a bounded, three-dimensional Euclidean domain in terms of eigenvalues of the Dirichlet Laplacian. The result complements recent inequalities between curl curl and Neumann Laplacian eigenvalues. The curl curl eigenvalues considered here correspond to the Maxwell eigenvalue problem with constant material parameters.
\end{abstract}

\maketitle

\section{Introduction}

In this note we are interested in the eigenvalue problem
\begin{align}\label{eq:curlCurl}
 \begin{cases} 
  \curl \curl u & \hspace{-2.5mm} = \alpha u \quad \,\, \text{in}~\Omega, \\
  \hfill \diver u & \hspace{-2.5mm} = 0 \quad \quad \text{in}~\Omega, \\
  \hfill u \times \nu & \hspace{-2.5mm} = 0 \quad \quad \text{on}~\partial \Omega,
 \end{cases}
\end{align}
on a bounded, connected Lipschitz domain $\Omega \subset \R^3$ with boundary $\partial \Omega$; $\nu$ denotes the exterior unit normal vector on $\partial \Omega$. This problem has been studied widely since it stems from the time-harmonic Maxwell equations in electromagnetism, see, e.g.,~\cite{LZ21}.

Let us denote by 
\begin{align*}
 \alpha_1 \leq \alpha_2 \leq \dots
\end{align*}
the eigenvalues of the problem \eqref{eq:curlCurl}, counted with multiplicities. We point out that all eigenvalues are non-negative and that an eigenvalue zero may occur; however, if $\Omega$ has no holes, i.e.\ $\partial \Omega$ is connected, then $\alpha_1 > 0$; cf.\ \cite[Chapter I, Lemma 3.4]{GR}. In recent years there has been an increasing interest in comparisons between the eigenvalues $\alpha_j$ and the eigenvalues of other equations on $\Omega$, such as the Stokes equation with Dirichlet boundary conditions or the Laplacian with Neumann boundary conditions \cite{P15,ZZ18,Z18}. For instance, denoting by $\mu_2$ the first non-zero eigenvalue of the Neumann Laplacian, Pauly \cite{P15} proved
\begin{align}\label{eq:Zhang}
 \mu_2 \leq \alpha_1
\end{align}
if $\Omega$ is convex. Zhang \cite{Z18} showed that
\begin{align}\label{eq:curlStokes}
 \alpha_j < \gamma_j, \quad j \in \N,
\end{align}
holds for the eigenvalues $\gamma_1 \leq \gamma_2 \leq \dots$ of the Stokes operator with Dirichlet boundary conditions if $\Omega$ is simply connected. 

In the present paper we compare the eigenvalues $\alpha_j$ with the eigenvalues of the Dirichlet Laplacian $- \Delta_{\rm D}$, denoted by
\begin{align*}
 0 < \lambda_1 < \lambda_2 \leq \lambda_3 \leq \dots,
\end{align*}
again counted according to their multiplicities. To the best of our knowledge, not much is known. Zhang \cite{Z18} shows $\lambda_1 < \gamma_1$ on $C^1$-domains, which does not allow to derive any conclusion about e.g.\ $\alpha_1$ when comparing with \eqref{eq:curlStokes}. However, in this note we prove the following result.

\begin{theorem}\label{thm:intro}
On any bounded, connected Lipschitz domain $\Omega \subset \R^3$, $\alpha_3 < \lambda_1$ and
\begin{align}\label{eq:resultIntro}
 \alpha_{2 k + 1} \leq \lambda_k, \quad k \in \N,
\end{align}
hold. If, in addition, $\Omega$ is a polyhedron, then the inequality \eqref{eq:resultIntro} is strict for all $k \in \N$.
\end{theorem}

Theorem \ref{thm:intro} is a consequence of Theorem \ref{thm:main} below, where slightly more general conditions for strictness of the inequality \eqref{eq:resultIntro} are provided. The question whether the inequality \eqref{eq:resultIntro} is always strict remains open. We point out that the inequality $\alpha_3 < \lambda_1$ is a significant improvement of Zhang's result $\alpha_1 < \gamma_1$, since $\lambda_1 < \gamma_1$.

Our proof of Theorem \ref{thm:intro} relies on a variational principle for the union of the Dirichlet Laplacian eigenvalues and the eigenvalues of the curl curl operator associated with the problem \eqref{eq:curlCurl}. This variational principle corresponds to a bilinear form studied earlier by Costabel and Dauge \cite{CD99}, see also Costabel \cite{C91}. However, it seems it has not been used before for the purpose of deriving eigenvalue inequalities of this kind.

Finally, we would like to point out that, in view of \eqref{eq:Zhang}, Theorem \ref{thm:intro} might provide a tool towards proving $\mu_{k + 3} \leq \lambda_k$, $k \in \N$, where $0 = \mu_1 < \mu_2 \leq \dots$ are the eigenvalues of the Neumann Laplacian. This is known to hold for convex $\Omega \subset \R^3$ \cite{LW86} and conjectured to hold for any bounded Lipschitz domain, see, e.g., \cite[Conjecture 3.2.42]{LMP}. The two-dimensional variant of this conjecture, $\mu_{k+2} \leq \lambda_k$ for all $k$, was recently established for all simply connected planar domains by the author \cite{R23p} using an approach similar to the one taken here.

This note is organized as follows. In Section 2 we review the rigorous definition of the eigenvalue problem \eqref{eq:curlCurl} and some properties of the involved function spaces and provide the necessary variational principles. Section 3 then contains the proof of the main result.

\section{Operator formulation and variational principles}\label{sec:prel}

In this section we collect all necessary tools for the proof of our main result. Throughout the whole paper we assume that $\Omega \subset \R^3$ is a bounded, connected Lipschitz domain. We define the spaces
\begin{align*}
 H (\diver; \Omega) = \left\{ u \in L^2 (\Omega)^3 : \diver u \in L^2 (\Omega) \right\}
\end{align*}
and
\begin{align*}
 H (\curl; \Omega) = \left\{ u \in L^2 (\Omega)^3 : \curl u \in L^2 (\Omega)^3 \right\},
\end{align*}
which are Hilbert spaces equipped with their natural norms. Next we recall that the mapping 
\begin{align*}
 C^\infty (\overline \Omega)^3 \ni u \mapsto u |_{\partial \Omega} \times \nu
\end{align*}
extends by continuity to a bounded linear operator from $H (\curl; \Omega)$ into $H^{- 1/2} (\partial \Omega)^3$; denoting the action of this operator by $u |_{\partial \Omega} \times \nu$ for general $u \in H (\curl; \Omega)$, the integration-by-parts formula
\begin{align}
\label{eq:PI}
 \int_\Omega \curl u \cdot v - \int_\Omega u \cdot \curl v = \left( u |_{\partial \Omega} \times \nu, v |_{\partial \Omega} \right)_{\partial \Omega}, \quad u \in H (\curl; \Omega), v \in H^1 (\Omega)^3,
\end{align}
holds, where $(\cdot, \cdot)_{\partial \Omega}$ is the duality between $H^{- 1/2} (\partial \Omega)^3$ and $H^{1/2} (\partial \Omega)^3$ and $v |_{\partial \Omega}$ denotes the componentwise trace of $v$, see \cite[Chapter I, Theorem 2.11]{GR}. We define
\begin{align*}
 H_0 (\curl; \Omega) := \left\{ u \in H (\curl; \Omega) : u |_{\partial \Omega} \times \nu = 0 \right\},
\end{align*}
the space of vector fields in $H (\curl; \Omega)$ which are normal on the boundary. For later use we note that
\begin{align}\label{eq:bdr0}
 \nabla H_0^1 (\Omega) := \left\{ \nabla \phi : \phi \in H_0^1 (\Omega) \right\} \subset H_0 (\curl; \Omega),
\end{align}
where $H_0^1 (\Omega)$ is the kernel of the trace operator on $H^1 (\Omega)$; this inclusion follows from \eqref{eq:PI} and another integration by parts.

We will also make use of the space
\begin{align*}
 H (\diver 0; \Omega) = \left\{ u \in L^2 (\Omega)^3 : \diver u = 0 \right\}.
\end{align*}
In fact, the orthogonal decomposition
\begin{align}\label{eq:orthoDec}
 L^2 (\Omega)^3 = \nabla H_0^1 (\Omega) \oplus H (\diver 0; \Omega)
\end{align}
holds. Indeed, it follows from Poincar\'e's inequality on $H_0^1 (\Omega)$ that $\nabla H_0^1 (\Omega)$ is a closed subspace of $L^2 (\Omega)^3$, and its orthogonal complement can be computed easily via Green's first identity.

Next let us recall the weak formulations of the eigenvalue problems we are interested in. We will make use of the following standard result \cite[Sections 4.4--4.5]{D95}.

\begin{proposition}\label{prop:abstractWeak}
Let $\cH$ be a real, infinite-dimensional Hilbert space with inner product $(\cdot, \cdot)$ and corresponding norm $\| \cdot \|$, let $V$ be a Hilbert space densely and compactly embedded into $\cH$, and let $b : V \times V \to \R$ be a symmetric bilinear form which is non-negative, i.e.\ 
\begin{align*}
 b (u, u) \geq 0, \quad u \in V.
\end{align*}
Assume, moreover, that $V$ equipped with the norm defined by
\begin{align*}
 \|u\|_b^2 := \|u\|^2 + b (u, u), \quad u \in V,
\end{align*}
is complete. Then the weak eigenvalue problem 
\begin{align}\label{eq:weakEV}
 b (u, v) = \lambda (u, v), \quad v \in V,
\end{align}
where $u \in V, u \neq 0$, has an infinite sequence of non-negative eigenvalues $\lambda_1 \leq \lambda_2 \leq \dots$ and corresponding eigenvectors $u_1, u_2, \dots$ which form an orthonormal basis of $\cH$. Moreover, the eigenvalues are given by
\begin{align*}
 \lambda_k = \min_{\substack{U \subset V \\ \dim U = k}} \, \max_{\substack{u \in U \\ u \neq 0}} \, \frac{b (u, u)}{\|u\|^2}
\end{align*}
for all $k \in \N$. Furthermore, there exists a unique self-adjoint operator $B$ in $\cH$ with domain $\dom B$ related to $b$ via $\dom B \subset V$ and
\begin{align*}
 (B u, v) = b (u, v), \quad u \in \dom B, v \in V.
\end{align*}
The eigenvalues $\lambda$ and eigenvectors $u$ of $B$ are precisely the solutions of \eqref{eq:weakEV}.
\end{proposition}

The eigenvalue problem \eqref{eq:curlCurl} corresponds to the choice $\cH = H (\diver 0; \Omega)$, $V = H_0 (\curl; \Omega) \cap H (\diver 0; \Omega)$ and
\begin{align*}
 b (u, v) = \int_\Omega \curl u \cdot \curl v, \quad u, v \in H_0 (\curl; \Omega) \cap H (\diver 0; \Omega),
\end{align*}
in Proposition \ref{prop:abstractWeak}. Therefore the weak formulation of \eqref{eq:curlCurl} is the following: a non-trivial $u \in H_0 (\curl; \Omega) \cap H (\diver 0; \Omega)$ is an eigenfunction of \eqref{eq:curlCurl} corresponding to the eigenvalue $\alpha$ if and only if
\begin{align}\label{eq:curlCurlWeak}
 \int_\Omega \curl u \cdot \curl v = \alpha \int_\Omega u \cdot v, \quad v \in H_0 (\curl; \Omega) \cap H (\diver 0; \Omega).
\end{align}
The corresponding strong formulation is the eigenvalue problem for the self-adjoint operator $C$ in $H (\diver 0; \Omega)$ given by
\begin{align}\label{eq:C}
\begin{split}
 C u & = \curl \curl u, \\
 \dom C & = \left\{ u \in H_0 (\curl; \Omega) \cap H (\diver 0; \Omega) : \curl \curl u \in L^2 (\Omega)^3 \right\}.
\end{split}
\end{align}
In other words, $u \in H_0 (\curl; \Omega) \cap H (\diver 0; \Omega)$ belongs to $\ker (C - \lambda)$ if and only if \eqref{eq:curlCurlWeak} holds.

Secondly, the Laplace eigenvalue problem with Dirichlet boundary conditions corresponds to the choices $\cH = L^2 (\Omega)$, $V = H_0^1 (\Omega)$ and
\begin{align*}
 b (\phi, \psi) = \int_\Omega \nabla \phi \cdot \nabla \psi, \quad \phi, \psi \in H_0^1 (\Omega),
\end{align*}
in Proposition \ref{prop:abstractWeak}. Its eigenvalues and eigenfunction belong to the self-adjoint operator 
\begin{align*}
 - \Delta_{\rm D} \phi & = - \Delta \phi, \quad \dom (- \Delta_{\rm D}) = \left\{\phi \in H_0^1 (\Omega) : \Delta u \in L^2 (\Omega) \right\}.
\end{align*}

The third object which will play an important role is the eigenvalue problem corresponding to the following choices in Proposition \ref{prop:abstractWeak}: we take $\cH = L^2 (\Omega)^3$, $V = H (\diver; \Omega) \cap H_0 (\curl; \Omega)$ and 
\begin{align}\label{eq:b}
 \sb (u, v) & = \int_\Omega \left( \diver u \, \diver v + \curl u \cdot \curl v \right), \quad u, v \in H (\diver; \Omega) \cap H_0 (\curl; \Omega).
\end{align}
It is not hard to see that $V$ is complete with the norm induced by $\sb$; this follows from the facts that $H_0 (\curl; \Omega)$ is a closed subspace of $H (\curl; \Omega)$, and that the norm $\| \cdot \|_\sb$ dominates the norms of $H (\diver; \Omega)$ and $H (\curl; \Omega)$. For the compactness of the embedding of $H (\diver; \Omega) \cap H_0 (\curl; \Omega)$ into $L^2 (\Omega)^3$ see \cite{W74}. We denote by $B$ the self-adjoint operator in $L^2 (\Omega)^3$ induced by $\sb$ as in Proposition \ref{prop:abstractWeak}.

\begin{remark}
By construction, any $u \in \dom B$ is normal on the boundary, $u |_{\partial \Omega} \times \nu = 0$, in the sense specified above. Moreover, a simple integration by parts yields that any vector field $u \in \dom B$ which is sufficiently regular up to the boundary of $\Omega$ satisfies $(\diver u) |_{\partial \Omega} = 0$. This indicates that the domain of $B$ consists of all sufficiently regular vector fields satisfying both boundary conditions. However, for the purpose of this article it is not necessary to compute the precise regularity of the operator domain, since all arguments will be carried out on the level of quadratic forms.
\end{remark}

It has been observed in \cite[Theorem 1.1]{CD99} that the eigenvalues and eigenfunctions corresponding to the bilinear form $\sb$ (and, thus, the operator $B$) are closely related to those of both the curl curl and the Dirichlet Laplacian problem. In order to make this note self-contained we provide a proof of the following statement.

\begin{proposition}\label{prop:spectrum}
Let $B$ be the self-adjoint operator in $L^2 (\Omega)^3$ associated with the bilinear form $\sb$ defined in \eqref{eq:b}. Then the eigenvalues $\eta$ of $B$ are given by the following two classes.
\begin{enumerate}
 \item $\eta$ is an eigenvalue of the Laplacian $- \Delta_{\rm D}$ in $L^2 (\Omega)$ with Dirichlet boundary conditions; or
 \item $\eta$ is an eigenvalue of the curl-curl operator $C$ in $H (\diver 0; \Omega)$.
\end{enumerate}
In fact, for each $\eta \in \R$,
\begin{align}\label{eq:kernel}
 \ker (B - \eta) = \nabla \ker \left( - \Delta_{\rm D} - \eta \right) \oplus \ker (C - \eta).
\end{align}
\end{proposition}

\begin{proof}
Let first $B u = \eta u$, i.e., $u \in \dom \sb = H (\diver; \Omega) \cap H_0 (\curl; \Omega)$ and
\begin{align}\label{eq:weakEv*}
 \sb (u, v) = \eta \int_\Omega u \cdot v, \quad v \in \dom \sb.
\end{align}
According to \eqref{eq:orthoDec} we can write $u = \nabla \phi + w$, where $\phi \in H_0^1 (\Omega)$ and $w \in L^2 (\Omega)^3$ with $\diver w = 0$. Then $\diver \nabla \phi = \diver u \in L^2 (\Omega)$ and $\nabla \phi \in H_0 (\curl; \Omega)$ by \eqref{eq:bdr0}, that is, $\nabla \phi \in \dom \sb$. From this we also conclude $w = u - \nabla \phi \in \dom \sb$. If we plug $v = \nabla \psi$ with $\psi \in \dom (- \Delta_{\rm D})$ into \eqref{eq:weakEv*}, then we get
\begin{align*}
 \int_\Omega \Delta \phi \, \Delta \psi = \sb (u, v) = \eta \int_\Omega u \cdot v = \eta \int_\Omega \nabla \phi \cdot \nabla \psi = - \eta \int_\Omega \phi \, \Delta \psi.
\end{align*}
Thus $\Delta \phi + \eta \phi$ is orthogonal to $\ran (- \Delta_{\rm D}) = L^2 (\Omega)$ and we have shown $\phi \in \ker (- \Delta_{\rm D} - \eta)$. On the other hand, for any $v \in \dom \sb$ with $\diver v = 0$, \eqref{eq:weakEv*} yields
\begin{align*}
 \int_\Omega \curl w \cdot \curl v = \sb (u, v) = \eta \int_\Omega w \cdot v,
\end{align*}
which is the weak formulation of the eigenvalue problem for the operator $C$; hence $w \in \ker (C - \eta)$.

Let us now take $\phi \in \ker (- \Delta_{\rm D} - \eta)$ and $w \in \ker (C - \eta)$. Then both $\nabla \phi$ and $w$ belong to $\dom \sb$, cf.\ \eqref{eq:bdr0}, and
\begin{align*}
 \sb (\nabla \phi + w, v) & = \int_\Omega \left( \Delta \phi \, \diver v + \curl w \cdot \curl v \right) = - \eta \int_\Omega \left( \phi \, \diver v - w \cdot v \right) \\
 & = \eta \int_\Omega \left( \nabla \phi + w \right) \cdot v
\end{align*}
holds for all $v \in \dom \sb$. Thus $\nabla \phi + w \in \ker (B - \eta)$.
\end{proof}

We formulate two consequences of Proposition \ref{prop:spectrum}. The first one follows immediately through Proposition \ref{prop:abstractWeak}.

\begin{corollary}
Denote by $\eta_1 \leq \eta_2 \leq \dots$ the eigenvalues of the operator $B$ in Proposition \ref{prop:spectrum}, i.e.\ the ordered sequence of all eigenvalues of the operators $C$ and $- \Delta_{\rm D}$, with multiplicities. Then
\begin{align}\label{eq:minMax}
 \eta_k = \min_{\substack{U \subset H (\diver; \Omega) \cap H_0 (\curl; \Omega) \\ \dim U = k}} \; \max_{\substack{u \in U \\ u \neq 0}} \frac{\int_\Omega \left( (\diver u)^2 + |\curl u|^2 \right)}{\int_\Omega |u|^2}
\end{align}
holds for all $k \in \N$. 
\end{corollary}

Secondly, we note the following property.

\begin{corollary}\label{cor:divBoundary}
Let $B$ be the operator defined in Proposition \ref{prop:spectrum} and $\eta \in \R$. Then $\diver u \in H_0^1 (\Omega)$ holds for each $u \in \ker (B - \eta)$.
\end{corollary}

\begin{proof}
By \eqref{eq:kernel}, each $u \in \ker (B - \eta)$ can be written $u = \nabla \phi + v$ for some $\phi \in \ker (- \Delta_{\rm D} - \eta)$ and $v \in \ker (C - \eta)$. Then $\diver u = \Delta \phi = - \eta \phi \in H_0^1 (\Omega)$.
\end{proof}

\section{curl-curl vs.\ Dirichlet Laplacian eigenvalues}

The main result of this note is the following theorem. We denote by $\lambda_1 < \lambda_2 \leq \dots$ the eigenvalues of the Dirichlet Laplacian $- \Delta_{\rm D}$ and by $\alpha_1 \leq \alpha_2 \leq \dots$ the eigenvalues of the operator $C$ in \eqref{eq:C}, all eigenvalues counted with multiplicities.

\begin{theorem}\label{thm:main}
Assume that $\Omega \subset \R^3$ is a bounded, connected Lipschitz domain. Then the inequality
\begin{align}\label{eq:weak}
 \alpha_{2 k + 1} \leq \lambda_k
\end{align}
holds for all $k \in \N$. If, in addition, $\lambda_k$ is a simple eigenvalue of $- \Delta_{\rm D}$ or $\partial \Omega$ contains three planar pieces whose normals are linearly independent, then
\begin{align}\label{eq:strong}
 \alpha_{2 k + 1} < \lambda_k
\end{align}
holds.
\end{theorem}

\begin{proof}
Let us fix some $k$ and let $\phi_1, \dots, \phi_k$ be an orthonormal set in $L^2 (\Omega)$ such that $- \Delta_{\rm D} \phi_j = \lambda_j \phi_j$, $j = 1, \dots, k$. In particular, $\phi_j |_{\partial \Omega} = 0$ for all $j$. Let
\begin{align*}
 u = \begin{pmatrix} u_1 \\ u_2 \\ u_3 \end{pmatrix} = \sum_{j = 1}^k \begin{pmatrix} \alpha_j \phi_j \\ \beta_j \phi_j \\ \gamma_j \phi_j \end{pmatrix},
\end{align*}
where $\alpha_j, \beta_j, \gamma_j \in \R$, $j = 1, \dots, k$, are arbitrary constants. Note that all these functions $u$ belong to $H (\diver; \Omega) \cap H_0 (\curl; \Omega)$ and form a subspace of dimension $3 k$. Moreover,
\begin{align}\label{eq:EVest}
 \int_\Omega |\nabla u_l|^2 \leq \lambda_k \int_\Omega |u_l|^2, \quad l = 1, 2, 3.
\end{align}
Using the bilinear form $\sb$ defined in \eqref{eq:b} we have
\begin{align*}
 \sb (u, u) & = \sum_{l = 1}^3 \int_\Omega |\nabla u_l|^2 + 2 \int_\Omega \big( \partial_1 u_1 \partial_2 u_2 + \partial_1 u_1 \partial_3 u_3 + \partial_2 u_2 \partial_3 u_3 \\
 & \hspace{43mm} - \partial_2 u_3 \partial_3 u_2 - \partial_3 u_1 \partial_1 u_3 - \partial_1 u_2 \partial_2 u_1 \big) \\
 & = \sum_{l = 1}^3 \int_\Omega |\nabla u_l|^2 \leq \lambda_k \sum_{l = 1}^3 \int_\Omega |u_l|^2 = \lambda_k \int_\Omega |u|^2,
\end{align*}
where we have used \eqref{eq:EVest} as well as integration by parts and $u_l \in H_0^1 (\Omega)$. Let, furthermore, $v \in \ker (B - \lambda_k)$ be arbitrary, where $B$ is the operator defined in Proposition \ref{prop:spectrum}; naturally, $v$ is trivial unless $\lambda_k$ is an eigenvalue of $B$. Then
\begin{align}\label{eq:dasIsses}
\begin{split}
 \sb (u + v, u + v) & = \sb (u, u) + 2 \sb (u, v) + \sb (v, v) = \sb (u, u) + 2 \int_\Omega u \cdot B v + \int_\Omega v \cdot B v \\
 & \leq \lambda_k \int_\Omega |u|^2 + 2 \lambda_k \int_\Omega u \cdot v + \lambda_k \int_\Omega |v|^2 = \lambda_k \int_\Omega |u + v|^2.
\end{split}
\end{align}
To conclude \eqref{eq:weak}, note first that 
\begin{align}\label{eq:einAnfang}
 H_0^1 (\Omega)^3 \cap \nabla \ker (- \Delta_{\rm D} - \lambda_k) = \{0\}.
\end{align}
Indeed, $\nabla \phi \in H_0^1 (\Omega)^3$ for some $\phi \in \ker (- \Delta_{\rm D} - \lambda_k)$ implies $\phi \in H^2 (\Omega) \cap H_0^1 (\Omega)$ and $\nu \cdot \nabla \phi |_{\partial \Omega} = 0$ and, by unique continuation, $\phi = 0$ identically; see, e.g., \cite[Lemma 2.2]{LR17}. As $\nabla \ker (- \Delta_{\rm D} - \lambda_k) \subset \ker (B - \lambda_k)$, see Proposition \ref{prop:spectrum}, from \eqref{eq:dasIsses}, \eqref{eq:einAnfang} and \eqref{eq:minMax} it follows
\begin{align*}
 \eta_{3 k + \dim \ker (- \Delta_{\rm D} - \lambda_k)} \leq \lambda_k.
\end{align*}
Thus the union of the eigenvalues of $C$ and $- \Delta_{\rm D}$ contains at least $3 k + \dim \ker (- \Delta_{\rm D} - \lambda_k)$ points in $[0, \lambda_k]$, counted with multiplicities. Since the number of eigenvalues of $- \Delta_{\rm D}$ in $[0, \lambda_k]$ is at most $k - 1 + \dim \ker (- \Delta_{\rm D} - \lambda_k)$, it follows that $C$ has at least
\begin{align*}
 3 k + \dim \ker (- \Delta_{\rm D} - \lambda_k) - \left( k - 1 + \dim \ker (- \Delta_{\rm D} - \lambda_k) \right) = 2 k + 1
\end{align*}
eigenvalues in $[0, \lambda_k]$ and the claim \eqref{eq:weak} follows.

Let us now assume that $\lambda_k$ is a simple eigenvalue of $- \Delta_{\rm D}$. Our aim is to show
\begin{align}\label{eq:yes}
 H_0^1 (\Omega)^3 \cap \ker (B - \lambda_k) = \{ 0 \},
\end{align}
which, together with \eqref{eq:dasIsses}, implies
\begin{align*}
 \eta_{3 k + \dim \ker (B - \lambda_k)} \leq \lambda_k;
\end{align*}
the latter yields $\eta_{3 k} < \lambda_k$, and as $- \Delta_{\rm D}$ has at most $k - 1$ eigenvalues in $[0, \lambda_k)$, \eqref{eq:strong} follows.

To show \eqref{eq:yes} under the simplicity assumption on $\lambda_k$, let $\phi \in \ker (- \Delta_{\rm D} - \lambda_k)$ be non-trivial and let $u \in H_0^1 (\Omega)^3 \cap \ker (B - \lambda_k)$. Then
\begin{align*}
 u = \phi \alpha = c \nabla \phi + v
\end{align*}
holds for some $\alpha \in \R^3$ and $c \in \R$, where $v \in \ker (C - \lambda_k)$; cf.\ \eqref{eq:kernel}. Moreover, we have
\begin{align*}
 0 = \diver v = \diver \left( \phi \alpha - c \nabla \phi \right) = \alpha \cdot \nabla \phi + c \lambda \phi
\end{align*}
in $\Omega$. Then
\begin{align}\label{eq:strangeEquation}
 0 & = \int_\Omega \left( \alpha \cdot \nabla \phi + c \lambda \phi \right)^2 = \int_\Omega (\alpha \cdot \nabla \phi)^2 + c^2 \lambda_k^2 \int_\Omega \phi^2,
\end{align}
where we have used
\begin{align*}
 \int_\Omega \phi \, \alpha \cdot \nabla \phi = \frac{1}{2} \int_\Omega \diver (\phi^2 \alpha) = 0
\end{align*}
by the divergence theorem. From \eqref{eq:strangeEquation} we immediately obtain $\alpha \cdot \nabla \phi = 0$ in $\Omega$. As $\phi$ vanishes on $\partial \Omega$ and is non-trivial, $\alpha = 0$ follows, that is, $u = 0$. We have shown \eqref{eq:yes}.

Now assume that $\partial \Omega$ contains three planar pieces $\Sigma$, $\Sigma'$, $\Sigma''$ with linearly independent normal vectors and that, again, $u \in H_0^1 (\Omega)^3 \cap \ker (B - \lambda_k)$. Let $S, T, N$ be an orthonormal basis of $\R^3$ such that $N = (N_1, N_2, N_3)^\top$ is equal to the constant outer unit normal vector of $\Sigma$. Then $S$ and $T$ are tangential on $\Sigma$, and
\begin{align*}
 \nabla u_l |_\Sigma = \left( S \partial_S u_l + T \partial_T u_l + N \partial_N u_l \right) |_\Sigma = N \partial_\nu u_l |_\Sigma, \quad l = 1, 2, 3,
\end{align*}
where we denote by $\partial_S$ and $\partial_T$ the directional derivatives in the directions of $S$ and $T$, respectively. Then, by Corollary \ref{cor:divBoundary},
\begin{align*}
 0 = \left( \diver u \right) |_\Sigma = \sum_{l = 1}^3 N_l \partial_\nu u_l |_\Sigma = \partial_\nu \left( u \cdot N \right) |_\Sigma.
\end{align*}
As $u \cdot N$ solves $- \Delta (u \cdot N) = \lambda_k (u \cdot N)$ and vanishes on $\partial \Omega$, unique continuation implies $u \cdot N = 0$ constantly in $\Omega$. Denoting by $N'$ and $N''$ the constant outer unit normal vectors of $\Sigma'$ and $\Sigma''$, respectively, the exact same reasoning gives $u \cdot N' = 0 = u \cdot N''$ constantly in $\Omega$, and as $N, N'$ and $N''$ form a basis of $\R^3$, $u = 0$ identically in $\Omega$ follows. This completes the proof.
\end{proof}

In the following we say that $\Omega \subset \R^3$ is a polyhedron if it is connected and its boundary consists of finitely many pieces of planes. Note that polyhedra are not necessarily convex.

\begin{corollary}
If $\Omega \subset \R^3$ is a bounded polyhedron, then
\begin{align*}
 \alpha_{2 k + 1} < \lambda_k
\end{align*}
holds for all $k \in \N$.
\end{corollary}

Since the first eigenvalue $\lambda_1$ of $- \Delta_{\rm D}$ is always a simple eigenvalue, we immediately obtain the following.

\begin{corollary}
Let $\Omega \subset \R^3$ be a bounded, connected Lipschitz domain. Then $\alpha_3 < \lambda_1$ holds.
\end{corollary}

\section*{Acknowledgements}

The author gratefully acknowledges financial support by the grant no.\ 2022-03342 of the Swedish Research Council (VR).

\end{document}